\documentclass{amsart}
\usepackage{amssymb, amsmath, latexsym, mathabx, dsfont,tikz}
\usepackage{multirow}
\usepackage{setspace}
\usepackage{enumerate}
\usepackage{diagbox}
\usepackage{stmaryrd}
\usepackage{makecell}
\usepackage{longtable}

\renewcommand{\baselinestretch}{\baselinestretch}
\renewcommand{\baselinestretch}{1.1}
\numberwithin{equation}{section}

\newtheorem{thm}{Theorem}[section]
\newtheorem{lem}[thm]{Lemma}

\newtheorem{prop}[thm]{Proposition}

\theoremstyle{definition}

\theoremstyle{remark}
\newtheorem{rmk}[thm]{Remark}
\newtheorem{exam}[thm]{Example}

\numberwithin{equation}{section}

\newcommand{\ra}{{\ \longrightarrow \ }}
\newcommand{\nra}{{\ \longarrownot\longrightarrow \ }}

\newcommand{\gen}{\text{gen}}
\newcommand{\ord}{\text{ord}}
\newcommand{\rk}{\text{rank}}
\newcommand{\n}{{\mathbb N}}
\newcommand{\z}{{\mathbb Z}}
\newcommand{\q}{{\mathbb Q}}
\newcommand{\Mod}[1]{\ (\mathrm{mod}\ #1 )}

\begin{document}
%%%%%%%%%%%%%%%%%%%%%%%%%%%%%%%%%%%%%%%%%%%%%%%%%%%%%%%%%%%%%%%%%%%%%%%%%%%%%
%%%%%%%%%%%%%%%%%%%%%%%%%%%%%%%%%%%%%%%%%%%%%%%%%%%%%%%%%%%%%%%%%%%%%%%%%%%%%
\title{A classification of regular diagonal quadratic forms}

\author{Mingyu Kim}
\address{Department of Mathematics, Sungkyunkwan University, Suwon 16419, Korea}
\email{kmg2562@skku.edu}
\thanks{This work was supported by the National Research Foundation of Korea(NRF) grant funded by the Korea government(MSIT) (NRF-2021R1C1C2010133).}

%%%%%%%%%%%%%%%%%%%%%%%%%%%%%%%%%%%%%%%%%%%%%%%%%%%%%%%%%%%%%%%%%%%%%%%%%%%%%
\begin{abstract}
A positive-definite integral quadratic form is called regular if it represents every positive integer which is locally represented.
In this article, we classify all regular diagonal quadratic forms of rank greater than 3.
\end{abstract}
%%%%%%%%%%%%%%%%%%%%%%%%%%%%%%%%%%%%%%%%%%%%%%%%%%%%%%%%%%%%%%%%%%%%%%%%%%%%%
\maketitle

%\nocite{*}
%%%%%%%%%%%%%%%%%%%%%%%%%%%%%%%%%%%%%%%%%%%%%%%%%%%%%%%%%%%%%%%%%%%%%%%%%%%%%
%%%%%%%%%%%%%%%%%%%%%%%%%%%%%%%%%%%%%%%%%%%%%%%%%%%%%%%%%%%%%%%%%%%%%%%%%%%%%
%%%%%%%%%%%%%%%%%%%%%%%%%%%%%%%%%%%%%%%%%%%%%%%%%%%%%%%%%%%%%%%%%%%%%%%%%%%%%
%%%%%%%%%%%%%%%%%%%%%%%%%%%%%%%%%%%%%%%%%%%%%%%%%%%%%%%%%%%%%%%%%%%%%%%%%%%%%
%%%%%%%%%%%%%%%%%%%%%%%%%%%%%%%%%%%%%%%%%%%%%%%%%%%%%%%%%%%%%%%%%%%%%%%%%%%%%
\section{Introduction}
%%%%%%%%%%%%%%%%%%%%%%%%%%%%%%%%%%%%%%%%%%%%%%%%%%%%%%%%%%%%%%%%%%%%%%%%%%%%%
%%%%%%%%%%%%%%%%%%%%%%%%%%%%%%%%%%%%%%%%%%%%%%%%%%%%%%%%%%%%%%%%%%%%%%%%%%%%%
%%%%%%%%%%%%%%%%%%%%%%%%%%%%%%%%%%%%%%%%%%%%%%%%%%%%%%%%%%%%%%%%%%%%%%%%%%%%%
%%%%%%%%%%%%%%%%%%%%%%%%%%%%%%%%%%%%%%%%%%%%%%%%%%%%%%%%%%%%%%%%%%%%%%%%%%%%%
%%%%%%%%%%%%%%%%%%%%%%%%%%%%%%%%%%%%%%%%%%%%%%%%%%%%%%%%%%%%%%%%%%%%%%%%%%%%%
For a positive integer $n$ and a positive-definite (classically) integral $k$-ary quadratic form
$$
f=f(x_1,x_2,\dots,x_k)=\sum_{1\le i,j\le k}a_{ij}x_ix_j\ \ (a_{ij}=a_{ji}\in \z),
$$
we consider the Diophantine equation
\begin{equation} \label{eq1}
f(x_1,x_2,\dots,x_k)=n
\end{equation}
and the congruence equation
\begin{equation} \label{eq2}
f(x_1,x_2,\dots,x_k)\equiv n\Mod N.
\end{equation}
Obviously, \eqref{eq2} is solvable over $\z$ if \eqref{eq1} is.
Note that the converse holds when $k=1$, i.e.,
for any positive integer $a$, if the congruence equation
$$
ax^2\equiv n\Mod N
$$
has an integer solution for every positive integer $N$, then $n=ax_0^2$ for some $x_0\in \z$.
However, the converse is no longer true when $k$ is greater than 1, and one may take $f(x,y)=x^2+11y^2$ and $n=3$ as a counterexample.
A positive-definite integral quadratic form for which the converse holds is called {\it regular}, following Dickson \cite{D}.
A problem of determining all regular quadratic forms has been studied in cases of $k=2,3,4$.
Under the assumption of generalized Riemann hypothesis, all regular binary quadratic forms were enumerated \cite{CHPT} (see also \cite{Ka} and \cite{V}).
The ternary case has been of particular interest.
Jones \cite{J} showed that there are exactly 102 regular diagonal ternary quadratic forms (see also \cite{JP}).
The finiteness of regular ternary quadratic forms was established by Watson \cite{W}.
The 913 candidates for regular ternary quadratic forms were given by Jagy, Kaplansky and Schiemann \cite{JKS}.
All but 22 candidates were already shown to be regular at that time.
Oh \cite{O1} proved the regularities of 8 forms.
The regularities of the remaining 14 forms were proved under the assumption of generalized Riemann hypothesis by Lemke Oliver \cite{L}.
Earnest \cite{E2} showed that there are infinitely many regular quaternary quadratic forms.
A classification of regular diagonal quaternary quadratic forms was carried out by B.M. Kim \cite{BMK} but the paper was never published till now.

Throughout, we always assume that every quadratic form is positive-definite and classically integral.
A quadratic form $f$ is called universal if it represents all positive integers.
Note that, by definition, any universal quadratic form is regular.
Ramanujan \cite{R} found all universal diagonal quaternary quadratic forms.
Later, Bhargava \cite{B} showed that any universal quadratic form contains a universal subform of rank less than or equal to 5.

We call a regular quadratic form $f$ {\it new} if no proper subform of $f$ represents all integers represented by $f$.
For example, a universal quaternary quadratic form $f=x^2+2y^2+3z^2+5w^2$ is new regular since no proper subform of $f$ is universal (see \cite{B}).
Recently, in \cite{KO}, Oh and the author proved that there is an absolute constant $C$ such that the rank of any new regular quadratic form is less than $C$.

We denote the set of all nonnegative integers represented by a quadratic form $g$ by $Q(g)$.
Let $f=f(x_1,x_2,\dots,x_k)$ be a diagonal $k$-ary quadratic form.
For $i=1,2,\dots,k$, we define a $(k-1)$-ary subform $\widehat{f}_i$ of $f$ by removing the $i$-th summand $a_ix_i^2$ from $f$ so that
$$
\widehat{f}_i(x_1,\dots,x_{i-1},x_{i+1},\dots,x_k)=a_1x_1^2+a_2x_2^2+\cdots+a_{i-1}x_{i-1}^2+a_{i+1}x_{i+1}^2+\cdots+a_kx_k^2.
$$
We call a regular diagonal $k$-ary quadratic form $f$ {\it minimal} if $Q(\widehat{f}_i)\subsetneq Q(f)$ for all $i=1,2,\dots,k$.
Note that any quaternary regular diagonal form is minimal.

By definition, every new regular diagonal quadratic form is minimal.
However, the converse in not true in general, e.g., a universal diagonal quaternary form $g=x^2+y^2+z^2+w^2$ is minimal while it is not new since $x^2+y^2+z^2+4w^2$ is a proper subform of $g$ which is also universal.

In the present paper, we find all minimal regular diagonal quadratic forms of rank greater than 3, and describe a way to get all regular diagonal quadratic forms from the minimal ones.
This gives a classification of all regular diagonal quadratic forms of rank greater than 3.
In the midst of the classification, we redid the enumeration of regular diagonal quaternary quadratic forms.
We found some minor typos and errors in the list of regular diagonal quaternary quadratic forms in \cite{BMK}.
For example, for any $d\in \{9,15\} \cup \{ 2^{2r}3 : r\in \n\}$, the quaternary quadratic form $2x^2+3y^2+6z^2+dw^2$ is regular, and thus they should be included.
On the other hand, some forms including $h(r)=2x^2+3y^2+9z^2+2^{r+1}3^2w^2$ ($r\in \n$) are not regular and thus should be excluded.
Note that $13\cdot 2^{s}$ is represented locally but not globally by $h(r)$, where $s$ is the greatest odd integer less than or equal to $r$.

In the following, we adopt the geometric language of quadratic spaces and lattices, generally following \cite{Ki} and \cite{OM}.
Let $R$ be the ring of rational integers $\z$ or the ring of $p$-adic integers $\z_p$ for a prime $p$.
An $R$-lattice $L=R\mathbf{x}_1+R\mathbf{x}_2+\cdots+R\mathbf{x}_k$ is a finitely generated free $R$-module equipped with a non-degenerate symmetric bilinear map $B : L\times L \to R$.
The corresponding Gram matrix $M_L$ with respect to the basis $\{\mathbf x_i\}_{i=1}^k$ is defined to be $(B(\mathbf{x}_i,\mathbf{x}_j))_{1\le i,j\le k}$, and the corresponding quadratic form $f_L$ is defined by
$$
f_L(x_1,x_2,\dots,x_k)=\sum_{1\le i,j\le k}B(\mathbf{x}_i,\mathbf{x}_j)x_ix_j.
$$
The discriminant of $L$, denoted $dL$, is the determinant of a Gram matrix of $L$.
The quadratic map $Q : L\to R$ associated to the symmetric bilinear map $B$ is given by $Q(\mathbf{x})=B(\mathbf{x},\mathbf{x})$ for $\mathbf{x}\in L$.
The set of all elements in $R$ represented by $L$ is denoted by $Q(L)$.
The ideal of $R$ generated by the set $\{B(\mathbf{x},\mathbf{y}) : \mathbf{x},\mathbf{y}\in L\}$ ($\{Q(\mathbf{x}) : \mathbf{x}\in L\}$) is called {\it the scale of $L$} ({\it the norm of $L$}) and denoted by $\mathfrak{s}L$ ($\mathfrak{n}L$, respectively).

Throughout, we always assume that every $\z$-lattice $L$ is positive-definite in the sense that a Gram matrix $M_L$ of $L$ is positive-definite, and is primitive in the sense that $\mathfrak{s}L=\z$.
If a $\z$-lattice $\z \mathbf{x}_1+\z \mathbf{x}_2+\cdots+\z \mathbf{x}_k$ is diagonal, i.e., $B(\mathbf{x}_i,\mathbf{x}_j)=0$ whenever $i\neq j$, then this $\z$-lattice will be simply denoted by $\langle Q(\mathbf{x}_1),Q(\mathbf{x}_2),\dots,Q(\mathbf{x}_k)\rangle$.
For a set $S$ of nonnegative integers and a $\z$-lattice $K$, we define $\psi(S,K)$ to be the smallest integer in $S$ which is not represented by $K$.
When $K$ represents all of $S$, then we define $\psi(S,K)=\infty$.
Clearly, a $\z$-lattice $L$ is regular if and only if $\psi(\n \cup \{0\},L)=\infty$.

Let $L$ be a $\z$-lattice and $p$ be a prime.
A $\z_p$-lattice $L_p$ is defined by $L_p=L\otimes \z_p$.
We let
$$
\z_p^{\times}=\z_p-p\z_p \ \ \text{and}\ \ (\z_p^{\times})^2=\{ \gamma^2 : \gamma \in \z_p^{\times}\}.
$$
If a nonnegative integer $n$ is represented by $L_q$ for all primes $q$, then we say that {\it $n$ is locally represented by $L$}.
It is well known that if $n$ is locally represented by $L$, then $n$ is represented by a $\z$-lattice $K$ in the genus of $L$ (see \cite[102:5]{OM}).
The genus of $L$ will be denoted by $\gen(L)$, and we denote by $Q(\gen(L))$ the set of all nonnegative integers which are locally represented by $L$.
Note that $Q(\gen(L))$ can be determined easily by using the results in \cite{OM2}.
For each odd prime $q$, we denote by $\Delta_q$ a non-square unit in $\z_q$, and $\left(\frac{\cdot}{q}\right)$ will denote the Legendre symbol.
Any unexplained notation and terminologies can be found in \cite{Ki} or \cite{OM}.

%%%%%%%%%%%%%%%%%%%%%%%%%%%%%%%%%%%%%%%%%%%%%%%%%%%%%%%%%%%%%%%%%%%%%%%%%%%%%
%%%%%%%%%%%%%%%%%%%%%%%%%%%%%%%%%%%%%%%%%%%%%%%%%%%%%%%%%%%%%%%%%%%%%%%%%%%%%
%%%%%%%%%%%%%%%%%%%%%%%%%%%%%%%%%%%%%%%%%%%%%%%%%%%%%%%%%%%%%%%%%%%%%%%%%%%%%
%%%%%%%%%%%%%%%%%%%%%%%%%%%%%%%%%%%%%%%%%%%%%%%%%%%%%%%%%%%%%%%%%%%%%%%%%%%%%
%%%%%%%%%%%%%%%%%%%%%%%%%%%%%%%%%%%%%%%%%%%%%%%%%%%%%%%%%%%%%%%%%%%%%%%%%%%%%
\section{Minimalities and stabilities of regular diagonal $\z$-lattices}
%%%%%%%%%%%%%%%%%%%%%%%%%%%%%%%%%%%%%%%%%%%%%%%%%%%%%%%%%%%%%%%%%%%%%%%%%%%%%
%%%%%%%%%%%%%%%%%%%%%%%%%%%%%%%%%%%%%%%%%%%%%%%%%%%%%%%%%%%%%%%%%%%%%%%%%%%%%
%%%%%%%%%%%%%%%%%%%%%%%%%%%%%%%%%%%%%%%%%%%%%%%%%%%%%%%%%%%%%%%%%%%%%%%%%%%%%
%%%%%%%%%%%%%%%%%%%%%%%%%%%%%%%%%%%%%%%%%%%%%%%%%%%%%%%%%%%%%%%%%%%%%%%%%%%%%
%%%%%%%%%%%%%%%%%%%%%%%%%%%%%%%%%%%%%%%%%%%%%%%%%%%%%%%%%%%%%%%%%%%%%%%%%%%%%

For $k=1,2,3,\dots$, we define a set
$$
\mathcal{N}(k)=\{  (a_1,a_2,\dots,a_k)\in \n^k : a_1\le a_2\le \cdots \le a_k\},
$$
and put $\mathcal{N}=\bigcup_{k=1}^{\infty}\mathcal{N}(k)$.
To each vector $\mathbf{a}=(a_1,a_2,\dots,a_k)\in \mathcal{N}$, a diagonal $\z$-lattice $L(\mathbf{a})=\langle a_1,a_2,\dots,a_k\rangle$ corresponds.
We define a partial order on $\mathcal{N}$ as follow.
For two elements $\mathbf{a}=(a_1,a_2,\dots,a_r)$  and $\mathbf{b}=(b_1,b_2,\dots,b_s)$ in $\mathcal{N}$, $\mathbf{a}\preceq \mathbf{b}$ if and only if $(a_i)_{1\le i\le r}$ is a subsequence of $(b_j)_{1\le j\le s}$ such that $Q(L(\mathbf{a}))=Q(L(\mathbf{b}))$.
We write $\mathbf{a}\prec \mathbf{b}$ when $\mathbf{a}\preceq \mathbf{b}$ and $\mathbf{a}\neq \mathbf{b}$.
A regular diagonal $\z$-lattice $\langle a_1,a_2,\dots,a_k\rangle$ $(a_1\le a_2\le \cdots \le a_k)$ is called {\it minimal} if its coefficients vector $(a_1,a_2,\dots,a_k)$ is a minimal element of the poset $(\mathcal{N},\preceq)$.

Let $K$ be a $\z$-lattice and $n$ be a positive integer.
We put $J=K\perp \langle n\rangle$.
If $Q(K)=Q(J)$, then we say that {\it $n$ is redundant in $J$ (or redundant to $K$)}.
We also say that {\it $J$ is obtained from $K$ by a redundant insertion of $n$}.
Then one may see from definition that any regular diagonal $\z$-lattice can be obtained from a minimal regular diagonal $\z$-lattice by (finitely many) redundant insertions.
The next two lemmas together give a criterion for redundancy of a positive integer to a given regular $\z$-lattice.

%%%%%%%%%%%%%%%%%%%%%%%%%%%%%%%%%%%%%%%%%%%%%%%%%%%%%%%%%%%%%%%%%%%%%%%%%%%%%

\begin{lem} \label{lemredodd}
Let $p$ be an odd prime and $L$ be a $\z_p$-lattice.
For $\gamma \in \z_p$,
$$
Q(L)=Q(L\perp \langle \gamma \rangle)\ \ \text{if and only if}\ \ \gamma \z_p \subset Q(L).
$$
\end{lem}

\begin{proof}
We prove the ``if" part first.
Assume that $\gamma \z_p\subset Q(L)$.
Note that
$$
Q(L\perp \langle \gamma \rangle)=\left\{Q(\mathbf{x})+\alpha^2\gamma : \mathbf{x}\in L,\ \alpha \in \z_p \right\}.
$$
Let $\mathbf{x}$ be a vector in $L$ and $\alpha$ be an element in $\z_p$.
If $\ord_p(Q(\mathbf{x}))<\ord_p(\gamma)$, then by Local Square Theorem \cite[63:1]{OM}, there is an $\epsilon \in \z_p^{\times}$ such that
$$
Q(\mathbf{x})+\alpha^2\gamma=\epsilon^2Q(\mathbf{x})=Q(\epsilon \mathbf{x})\in Q(L).
$$
If $\ord_p(Q(\mathbf{x}))\ge \ord_p(\gamma)$, then we have
$$
Q(\mathbf{x})+\alpha^2\gamma \in \gamma \z_p \subset Q(L).
$$
Thus we have
$$
Q(L)\subseteq Q(L\perp \langle \gamma \rangle)=\{ Q(\mathbf{x})+\alpha^2\gamma : \mathbf{x}\in L,\ \alpha \in \z_p\}  \subseteq Q(L),
$$
which proves the ``if" part.

To show the ``only if" part, we assume that $Q(L)=Q(L\perp \langle \gamma \rangle)$.
For each $k=1,2,3,\dots$, we define a $\z$-lattice $K(n)$ by
$$
K(n)=L\perp \langle \gamma,\gamma,\dots,\gamma \rangle,
$$
where $\gamma$ is repeated $n$ times so that $\text{rank}(K(n))=\text{rank}(L)+n$.
Note that, for any $n\in \n$,
$$
Q(K(n+1))=Q(K(n)\perp \langle \gamma \rangle)=Q(K(n))+Q(\langle \gamma \rangle).
$$
Now one may use induction on $n$ to see that $Q(K(n))=Q(L)$.
Consequently, we have
$$
n\gamma \in Q(K(n))=Q(L)
$$
for every positive integer $n$.
In particular, we have $\{ \gamma, a\gamma, p\gamma, pa\gamma \} \subset Q(L)$, where $a$ is any positive integer which is a quadratic nonresidue modulo $p$.
From this follows that $\gamma \z_p\subset Q(L)$.
This completes the proof.
\end{proof}

%%%%%%%%%%%%%%%%%%%%%%%%%%%%%%%%%%%%%%%%%%%%%%%%%%%%%%%%%%%%%%%%%%%%%%%%%%%%%

\begin{lem} \label{lemred2}
Let $L$ be a $\z_2$-lattice.
For $\gamma \in \z_2$, $Q(L)=Q(L\perp \langle \gamma \rangle)$ if and only if all of the followings hold;
\begin{enumerate} [(i)]
\item $\{ s\in \{1,3,5,7\} : 2^{\ord_2(\gamma)-2}s\in Q(L)\}=\emptyset$, $\{1,5\}$, $\{3,7\}$ or $\{1,3,5,7\}$;
\item $\{ s\in \{1,3,5,7\} : 2^{\ord_2(\gamma)-1}s\in Q(L)\}=\emptyset$ or $\{1,3,5,7\}$;
\item $\gamma \z_2\subset Q(L)$.
\end{enumerate}
\end{lem}

\begin{proof}
The proof is quite straightforward and left as an exercise to the reader.
\end{proof}

%%%%%%%%%%%%%%%%%%%%%%%%%%%%%%%%%%%%%%%%%%%%%%%%%%%%%%%%%%%%%%%%%%%%%%%%%%%%%

\begin{exam} \label{ex148144}
For $r=1,2,\dots$, let $L(r)=\langle 1,48,144,2^4 3^{2r}\rangle$.
Note that $L(r)$ is a regular diagonal quaternary $\z$-lattice for any positive integer $r$.
For a prime $p$, one may easily check that
$$
Q(L_p)=\begin{cases} (\z_2^{\times})^2 \cup 4(\z_2^{\times})^2 \cup 20(\z_2^{\times})^2 \cup 16\z_2&\text{if}\ \ p=2,\\
\z_3-\left( 2(\z_3^{\times})^2 \cup \bigcup_{i=1}^{r}2\cdot 3^{2i-1}(\z_3^{\times})^2\right)&\text{if}\ \ p=3,\\
\z_p&\text{if}\ \ p\ge 5.\end{cases}
$$
By Lemmas \ref{lemredodd} and \ref{lemred2}, one may see that a positive integer $n$ is redundant to $L(r)$ if and only if $n$ is divisible by $2^43^{2r}$.
Hence for any vector $(b_5,b_6,\dots,b_k)\in \mathcal{N}(k-4)$, the $k$-ary $\z$-lattice
$$
\langle 1,48,144,2^4 3^{2r},2^4 3^{2r}b_5,2^4 3^{2r}b_6,\dots,2^4 3^{2r}b_k\rangle
$$
is a regular diagonal $\z$-lattice obtained from $L(r)$ by redundant insertions.
\end{exam}

%%%%%%%%%%%%%%%%%%%%%%%%%%%%%%%%%%%%%%%%%%%%%%%%%%%%%%%%%%%%%%%%%%%%%%%%%%%%%

We turn to the stability of regular diagonal $\z$-lattices.
For a positive integer $m$ and a $\z$-lattice $L$, let
$$
\Lambda_m(L)=\{ \mathbf{x}\in L : Q(\mathbf{x}+\mathbf{z})\equiv Q(\mathbf{z}) \Mod m \ \text{for all}\ \mathbf{z}\in L \}.
$$
Then $\Lambda_m(L)$ is a sublattice of $L$ and called {\it Watson transformation of $L$ modulo $m$}.
The primitive $\z$-lattice obtained from $\Lambda_m(L)$ by scaling the quadratic space $L\otimes \q$ by a suitable rational number will be denoted by $\lambda_m(L)$.
For more information and properties on this transformation, we refer the reader to \cite{CE},\cite{CO},\cite{W} or \cite{W2}.

%%%%%%%%%%%%%%%%%%%%%%%%%%%%%%%%%%%%%%%%%%%%%%%%%%%%%%%%%%%%%%%%%%%%%%%%%%%%%

\begin{lem} \label{lemlambda}
Let $p$ be a prime and let $L=\langle a_1,p^{e_2}a_2,p^{e_3}a_3,\dots,p^{e_k}a_k\rangle$ be a regular diagonal $\z$-lattice of rank $k\ge 4$ with $0\le e_2\le e_3\le \cdots \le e_k$ such that $(p,a_1a_2\cdots a_k)=1$.
Then the followings hold:
\begin{enumerate} [(i)]
\item If $p=2$ and $e_2\ge 1$, then $\lambda_2(L)$ is also a regular diagonal $\z$-lattice.
\item If $p=2$, $e_2=0$, $e_3\ge 2$ and $a_1\equiv a_2\Mod 4$, then $\lambda_4(L)$ is regular diagonal.
\item If $p=2$, $e_2=e_3=0$, $e_4\ge 2$ and $a_1\equiv a_2\equiv a_3\Mod 4$, then $\lambda_4(L)$ is regular diagonal.
\item If $p\neq 2$ and $e_2\ge 1$, then $\lambda_p(L)$ is regular diagonal.
\item If $p\neq 2$, $e_2=0$, $e_3\ge 1$ and $\left( \dfrac{-a_1a_2}{p}\right)=-1$, then $\lambda_p(L)$ is regular diagonal.
\end{enumerate}
\end{lem}

\begin{proof}
Since the proofs are quite similar to each other, we only provide the proof for the case (iii).
It suffices to show that $\Lambda_4(L)$ is regular and diagonal.
Let $\{ \mathbf{v}_1,\mathbf{v}_2,\dots,\mathbf{v}_k\}$ be a basis for $L$ with respect to which
$$
L=\z \mathbf{v}_1+\z \mathbf{v}_2+\cdots+\z \mathbf{v}_k\simeq \langle a_1,p^{e_2}a_2,p^{e_3}a_3,\dots,p^{e_k}a_k\rangle.
$$
One may easily check that
\begin{align} \label{eqlemlambda}
\begin{split}
\Lambda_4(L)&=\{ \mathbf{x}\in L : Q(\mathbf{x})\equiv 0\Mod 4,\ B(\mathbf{x},\mathbf{z})\equiv 0\Mod 2\ \text{for all}\ \mathbf{z}\in L\} \\
&=\{ b_1\mathbf{v}_1+b_2\mathbf{v}_2+\cdots+b_k\mathbf{v}_k : b_i\in \z,\ b_1\equiv b_2\equiv b_3\equiv 0\Mod 2\}.
\end{split}
\end{align}
Hence $\Lambda_4(L)\simeq \langle 4a_1,4a_2,4a_3,2^{e_4}a_4,2^{e_5}a_5,\dots,2^{e_k}a_k\rangle$.
To prove the regularity of $\Lambda_4(L)$, let $n$ be a positive integer which is represented by $\Lambda_4(L)$ over $\z_p$ for every prime $p$.
Then $n$ is represented by $L_p$ for every prime $p$ since $\Lambda_4(L)_p\subseteq L_p$.
By the regularity of $L$, $n$ is represented by $L$ and thus there is a vector $(c_1,c_2,\dots,c_k)\in \z^k$ such that
$$
n=c_1^2 a_1+c_2^2 a_2+c_3^2 a_3+c_4^2 2^{e_4}a_4+c_5^2 2^{e_5}a_5+\cdots+c_k^2 2^{e_k}a_k.
$$
Note that $n$ is divisible by 4 since $n$ is represented by $\Lambda_4(L)_2$ of which the norm is $4\z$.
Since $a_1\equiv a_2\equiv a_3\Mod 4$, we have $c_1\equiv c_2\equiv c_3\equiv 0\Mod 2$.
Hence $n$ is represented by $\Lambda_4(L)$ by Equation \ref{eqlemlambda}.
This proves the regularity of $\Lambda_4(L)$ and we have (iii).
\end{proof}

%%%%%%%%%%%%%%%%%%%%%%%%%%%%%%%%%%%%%%%%%%%%%%%%%%%%%%%%%%%%%%%%%%%%%%%%%%%%%

Let $L$ be a regular diagonal $\z$-lattice of rank $\ge 4$.
We say that {\it $L$ is 2-stable} if
$$
\langle 1,3\rangle \ra L_2\ \ \text{or}\ \ \langle 1,7\rangle \ra L_2.
$$
For an odd prime $p$, we say that {\it $L$ is $p$-stable} if
$$
\langle 1,-1\rangle \ra L_p\ \ \text{or}\ \ L_p\simeq \langle 1,-\Delta_p,p,-p\Delta_p\rangle.
$$
Note that $L$ is $p$-stable for any prime not dividing $2dL$.
Let $T$ be any set of primes.
We define $P(L)$ to be the set of all primes dividing $2dL$.
Then by taking $\lambda_q$-transformations for some suitable values $q$ with
$$
q\in (P(L)\cap T) \cup \{4\},
$$
we obtain a regular diagonal $\z$-lattice $L'$ which is $p$-stable for every prime $p$ in $T$ by using Lemma \ref{lemlambda}.

%%%%%%%%%%%%%%%%%%%%%%%%%%%%%%%%%%%%%%%%%%%%%%%%%%%%%%%%%%%%%%%%%%%%%%%%%%%%%
\begin{lem} \label{lemstabler}
Let $p$ be a prime and $L$ be a $p$-stable regular diagonal $\z$-lattice of rank $\ge 4$.
Then we have
\begin{enumerate} [(i)]
\item $\{ \gamma \in \z_2 : \ord_2(\gamma)\equiv 0\Mod 2 \} \subseteq Q(L_2)$ if $p=2$,
\item $Q(L_p)=\z_p$ if $p\neq 2$.
\end{enumerate}
\end{lem}

\begin{proof}
It follows immediately from the definition of $p$-stability.
\end{proof}

%%%%%%%%%%%%%%%%%%%%%%%%%%%%%%%%%%%%%%%%%%%%%%%%%%%%%%%%%%%%%%%%%%%%%%%%%%%%%

Let $P$ denote the set of all primes.
For a prime $p$, we define two sets $A_+(p)$ and $A_-(p)$ by
$$
A_{\pm}(p)=\left\{ q\in P : q\le p,\ \left(\frac{q}{p}\right)=\pm 1\right\},
$$
and put $A'_{\pm}(p)=A_{\pm}(p)-\{2\}$.

%%%%%%%%%%%%%%%%%%%%%%%%%%%%%%%%%%%%%%%%%%%%%%%%%%%%%%%%%%%%%%%%%%%%%%%%%%%%%

\begin{lem} \label{lemunstable}
Let $L$ be a regular diagonal $\z$-lattice of rank $\ge 4$.
Then $L$ is $p$-stable for every prime $p$ greater than 5.
\end{lem}

\begin{proof}
Suppose that $L$ is not $p$-stable for a prime $p$ greater than 5.
By Lemma \ref{lemlambda}, we may assume that $L$ is $q$-stable for all primes $q$ except $p$.
We write
$$
L=\langle a_1,p^{e_2}a_2,p^{e_3}a_3,\dots,p^{e_k}a_k\rangle,\ \ (p,a_1a_2\cdots a_k)=1\ \ \text{and}\ \ 0\le e_2\le e_3\le \cdots \le e_k.
$$
We define a set $H$ by
$$
H=\{ n\in Q(\gen(L)) :  1\le n\le p-1\}.
$$
Note that $H=\{ n\in Q(L) : 1\le n\le p-1\}$ since $L$ is regular.

Assume that $e_2=0$.
Since $L$ is not $p$-stable, we have $0<e_3$, and this implies that $p\le p^{e_i}a_i$ for any $i=3,4,\dots,k$. 
Hence we have $H\subset Q(\langle a_1,a_2\rangle)$.
Since $p>5$, we have
$$
\{1,3,5\} \subset H\subset Q(\langle a_1,a_2\rangle)
$$
by Lemma \ref{lemstabler}.
However, one may easily check that there does not exist binary $\z$-lattice $K$ satisfying $\{1,3,5\}\subset Q(K)$.

Thus we may further assume that $e_2>0$.
Then one may see that $H\subset Q(\langle a_1\rangle)$.
This implies that
\begin{equation} \label{equn}
\vert H\vert=\vert \{ x_1\in \n : a_1x_1^2\le p-1\} \vert \le \left[\sqrt{p-1}\right],
\end{equation}
where $[\cdot]$ is the greatest integer function.
On the other hand, we have
$$
H=\{ 1\le n\le p-1\}-(\{ 1\le n\le p-1 : n\nra L_2\} \cup \{1\le n\le p-1 : n\nra L_p\}),
$$
since $Q(L_q)=\z_q$ for any odd prime $q\neq p$ by Lemma \ref{lemstabler}(ii).
Since $L$ is 2-stable, we have
$$
\{ 1\le n\le p-1 : n\nra L_2\} \subseteq \{ 1\le n\le p-1 : n\equiv 0,2,6\Mod 8\}
$$
by Lemma \ref{lemstabler}(i).
Thus we have
$$
\vert \{ 1\le n\le p-1 : n\nra L_2\} \vert \le 3\left\lceil \frac{p-1}{8}\right\rceil,
$$
where $\lceil \cdot \rceil$ is the ceiling function.
Note that
$$
\vert \{ 1\le n\le p-1 : n\nra L_p\} \vert=\frac{p-1}{2}.
$$
It follows that
$$
\vert H \vert \ge p-1-3\left\lceil \frac{p-1}{8}\right\rceil-\frac{p-1}{2}.
$$
From this and \eqref{equn}, we have
$$
p-1-\frac{p-1}{2}-3\left\lceil \frac{p-1}{8}\right\rceil \le \vert H\vert \le \left[\sqrt{p-1}\right],
$$
which implies that $p\le 83$.
Assume further that $a_1\ge 2$.
Then $1\nra L$.
Since $L$ is regular and $1\ra L_q$ for any $q\neq p$, we must have $1\nra L_p$.
Hence the unimodular component of a Jordan decomposition of $L_p$ is isometric to $\langle \Delta_p\rangle$ and it follows that $\left(\dfrac{a_1}{p}\right)=-1$.
In this case, we show that $\vert A'_-(p)\vert \le 1$.
Take any prime $q$ in $A'_-(p)$.
Then one may easily check that $q\ra \gen(L)$.
Hence $q\in H$ and thus $q=a_1x^2$ for some $x\in \n$.
It follows that $q=a_1$.
This shows that $A'_-(p)\subseteq \{a_1\}$ and thus we have $\vert A'_-(p)\vert \le 1$.
One may directly check that
$$
\{ p\in P: 3\le p\le 83,\ \vert A'_-(p)\vert \le 1\}=\{3,5,11\}.
$$
Note that $A'_-(11)=\{7\}$.
Now, if $p=11$, then one may easily deduce that
$$
13\ra \gen(\langle 7,11^{e_2}a_2,11^{e_3}a_3,\dots,11^{e_k}a_k\rangle),
$$
whereas 13 cannot be represented by $\langle 7,11^{e_2}a_2,11^{e_3}a_3,\dots,11^{e_k}a_k\rangle$.
This is absurd and we have $p\in \{3,5\}$ when $a_1\ge 2$.
Assume that $a_1=1$.
Then one may easily deduce that $\vert A'_+(p)\vert=0$.
One may directly check that
$$
\{ p\in P : 3\le p\le 83,\ \vert A'_+(p)\vert =0\} =\{3,5,7\},
$$
and thus we have $p\in \{3,5,7\}$.
Suppose that $p=7$.
Then one may easily deduce that
$$
11\ra L=\langle 1,7^{e_2}a_2,7^{e_3}a_3,\dots,7^{e_k}a_k\rangle.
$$
Thus we may assume that $7^{e_2}a_2=7$ after rearrangement, if necessary.
Then one may easily deduce that $15\ra L$ and we may assume that $7^{e_3}a_3\in \{7,14\}$ since $15\nra \langle 1,7\rangle$.
Thus we have
$$
L=\langle 1,7,7,7^{e_4}a_4,\dots,7^{e_k}a_k\rangle \ \ \text{or}\ \ L=\langle 1,7,14,7^{e_4}a_4,\dots,7^{e_k}a_k\rangle.
$$
In both cases, we have $Q(L_2)=\z_2$.
It follows that $2\ra L_r$ for every prime $r$.
Since $L$ is regular, we have $2\ra L$.
This is absurd, and we have $p\in \{3,5\}$ when $a_1=1$.
This completes the proof.
\end{proof}

%%%%%%%%%%%%%%%%%%%%%%%%%%%%%%%%%%%%%%%%%%%%%%%%%%%%%%%%%%%%%%%%%%%%%%%%%%%%%

\begin{rmk}
Note that, for a prime $p$, the followings are direct consequences of the results in \cite{G} and \cite{P}.
\begin{enumerate} [(i)]
\item $\vert A'_+(p)\vert \ge 1$ if $p\ge 11$,
\item $\vert A'_-(p)\vert \ge 2$ if $p\ge 13$.
\end{enumerate}
\end{rmk}

%%%%%%%%%%%%%%%%%%%%%%%%%%%%%%%%%%%%%%%%%%%%%%%%%%%%%%%%%%%%%%%%%%%%%%%%%%%%%

\begin{lem} \label{lembound}
Let $L=\langle a_1,a_2,\dots,a_k\rangle$ be a regular diagonal $\z$-lattice with $a_1\le a_2\le \cdots \le a_k$ and let $n$ be a positive integer represented by $L_q$ for $q=2,3,5$.
If $n$ is not represented by $\langle a_1,a_2,\dots,a_i\rangle$ for some $1\le i\le k-1$, then we have $a_{i+1}\le n$.
\end{lem}

\begin{proof}
By Lemma \ref{lemstabler}(ii) and Lemma \ref{lemunstable}, $n$ is represented by $L_q$ for every prime $q$ greater than 5.
From this and the assumption, it follows that $n$ is locally represented by $L$.
If $a_{i+1}>n$, then $n$ cannot be represented by $L$.
Hence we have $a_{i+1}\le n$.
\end{proof}

%%%%%%%%%%%%%%%%%%%%%%%%%%%%%%%%%%%%%%%%%%%%%%%%%%%%%%%%%%%%%%%%%%%%%%%%%%%%%
%%%%%%%%%%%%%%%%%%%%%%%%%%%%%%%%%%%%%%%%%%%%%%%%%%%%%%%%%%%%%%%%%%%%%%%%%%%%%
%%%%%%%%%%%%%%%%%%%%%%%%%%%%%%%%%%%%%%%%%%%%%%%%%%%%%%%%%%%%%%%%%%%%%%%%%%%%%
%%%%%%%%%%%%%%%%%%%%%%%%%%%%%%%%%%%%%%%%%%%%%%%%%%%%%%%%%%%%%%%%%%%%%%%%%%%%%
%%%%%%%%%%%%%%%%%%%%%%%%%%%%%%%%%%%%%%%%%%%%%%%%%%%%%%%%%%%%%%%%%%%%%%%%%%%%%
\section{Ternary sections of regular diagonal $\z$-lattices}
%%%%%%%%%%%%%%%%%%%%%%%%%%%%%%%%%%%%%%%%%%%%%%%%%%%%%%%%%%%%%%%%%%%%%%%%%%%%%
%%%%%%%%%%%%%%%%%%%%%%%%%%%%%%%%%%%%%%%%%%%%%%%%%%%%%%%%%%%%%%%%%%%%%%%%%%%%%
%%%%%%%%%%%%%%%%%%%%%%%%%%%%%%%%%%%%%%%%%%%%%%%%%%%%%%%%%%%%%%%%%%%%%%%%%%%%%
%%%%%%%%%%%%%%%%%%%%%%%%%%%%%%%%%%%%%%%%%%%%%%%%%%%%%%%%%%%%%%%%%%%%%%%%%%%%%
%%%%%%%%%%%%%%%%%%%%%%%%%%%%%%%%%%%%%%%%%%%%%%%%%%%%%%%%%%%%%%%%%%%%%%%%%%%%%
Recall that, for a subset $S\subseteq \n \cup \{0\}$ and a $\z$-lattice $K$,
$$
\psi(S,K)=\begin{cases}\min(S-Q(K))&\text{if}\ \ S-Q(K)\neq \emptyset,\\
\infty&\text{otherwise}.\end{cases}
$$
Define $A$ to be the Cartesian product
$$
A=\{1,3,5,7\} \times \{1,2\} \times \{1,2\}.
$$
For each triple $(\delta_2,\delta_3,\delta_5)\in A$, we put
$$
S(\delta_2,\delta_3,\delta_5)=\{ n\in \n : n\equiv \delta_2\Mod 8,\ n\equiv \delta_3\Mod 3,\ n\equiv \pm \delta_5\Mod 5\},
$$
and denote the $i$-th smallest element in $S(\delta_2,\delta_3,\delta_5)$ by $s_i(\delta_2,\delta_3,\delta_5)$ so that
$$
S(\delta_2,\delta_3,\delta_5)=\{s_1(\delta_2,\delta_3,\delta_5)<s_2(\delta_2,\delta_3,\delta_5)<\cdots<s_i(\delta_2,\delta_3,\delta_5)<\cdots \}.
$$
For example,
$$
S(1,1,1)=\{1<49<121<169<241<289<\cdots \},
$$
and $s_1(1,1,1)=1$, $s_2(1,1,1)=49$, etc.
Define
$$
U_1(\delta_2,\delta_3,\delta_5)=\{ b_1\in \n : b_1\le s_1(\delta_2,\delta_3,\delta_5)\},
$$
and
$$
U_2(\delta_2,\delta_3,\delta_5)=\{ (b_1,b_2)\in \n^2 : b_1\in U_1(\delta_2,\delta_3,\delta_5),\ b_1\le b_2\le \psi(S(\delta_2,\delta_3,\delta_5),\langle b_1\rangle) \}.
$$
Let $(b_1,b_2)$ be a vector in $U_2(\delta_2,\delta_3,\delta_5)$.
For each prime $q\in \{2,3,5\}$, we define a set $T(\delta_2,\delta_3,\delta_5,b_1,b_2,q)$ to be
$$
\begin{cases} \{ n\in \n : n\ra \langle b_1,b_2,\delta_q\rangle \ \ \text{over}\ \ \z_q \}&\text{if}\ \ \left( \dfrac{\delta_q}{q}\right)\neq \left( \dfrac{b_i}{q}\right) \ \text{for all}\ i=1,2,\\
\{ n\in \n : n\ra \langle b_1,b_2\rangle \ \ \text{over}\ \ \z_q \} &\text{otherwise},\end{cases}
$$
and put $T(\delta_2,\delta_3,\delta_5,b_1,b_2)=\bigcap_{q\in \{2,3,5\}} T(\delta_2,\delta_3,\delta_5,b_1,b_2,q)$.
We define $U_3(\delta_2,\delta_3,\delta_5)$ to be the set
$$
\{ (b_1,b_2,b_3)\in \n^3 : (b_1,b_2)\in U_2(\delta_2,\delta_3,\delta_5),\ b_2\le b_3\le \psi(T(\delta_2,\delta_3,\delta_5,b_1,b_2),\langle b_1,b_2\rangle)\}.
$$
Put $U_3=\bigcup_{(\delta_2,\delta_3,\delta_5)\in A} U_3(\delta_2,\delta_3,\delta_5)$.
Note that $\vert U_3\vert=103$ and all diagonal ternary $\z$-lattices $\langle a_1,a_2,a_3\rangle$ for which the vector $(a_1,a_2,a_3)$ is in $U_3$ are listed in Table \ref{tableu3}.

%%%%%%%%%%%%%%%%%%%%%%%%%%%%%%%%%%%%%%%%%%%%%%%%%%%%%%%%%%%%%%%%%%%%%%%%%%%%%

\begin{table}[ht]
\caption{Ternary $\z$-lattices $\langle a_1,a_2,a_3\rangle$ for which $(a_1,a_2,a_3)\in U_3$}
\vskip -10pt
\begin{tabular}{|llll|}
\hline \rule[-2mm]{-1mm}{6mm}
1. $\langle 1,1,1\rangle$ &
2. $\langle 1,1,2\rangle$ &
3. $\langle 1,1,3\rangle$ &
4. $\langle 1,1,4\rangle$ \\
\hline \rule[-2mm]{-1mm}{6mm}
5. $\langle 1,1,5\rangle$ &
6. $\langle 1,1,6\rangle$ &
7. $\langle 1,1,8\rangle$ &
8. $\langle 1,1,9\rangle$ \\
\hline \rule[-2mm]{-1mm}{6mm}
9. $\langle 1,1,12\rangle$ &
10. $\langle 1,1,16\rangle$ &
11. $\langle 1,1,24\rangle$ &
12. $\langle 1,2,2\rangle$ \\
\hline \rule[-2mm]{-1mm}{6mm}
13. $\langle 1,2,3\rangle$ &
14. $\langle 1,2,4\rangle$ &
15. $\langle 1,2,5\rangle$ &
16. $\langle 1,2,6\rangle$ \\
\hline \rule[-2mm]{-1mm}{6mm}
17. $\langle 1,2,8\rangle$ &
18. $\langle 1,2,10\rangle$ &
19. $\langle 1,2,16\rangle$ &
20. $\langle 1,2,32\rangle$ \\
\hline \rule[-2mm]{-1mm}{6mm}
21. $\langle 1,3,3\rangle$ &
22. $\langle 1,3,4\rangle$ &
23. $\langle 1,3,6\rangle$ &
24. $\langle 1,3,9\rangle$ \\
\hline \rule[-2mm]{-1mm}{6mm}
25. $\langle 1,3,10\rangle$ &
26. $\langle 1,3,12\rangle$ &
27. $\langle 1,3,18\rangle$ &
28. $\langle 1,3,30\rangle$ \\
\hline \rule[-2mm]{-1mm}{6mm}
29. $\langle 1,3,36\rangle$ &
30. $\langle 1,4,4\rangle$ &
31. $\langle 1,4,6\rangle$ &
32. $\langle 1,4,8\rangle$ \\
\hline \rule[-2mm]{-1mm}{6mm}
33. $\langle 1,4,12\rangle$ &
34. $\langle 1,4,16\rangle$ &
${}^\dag$35. $\langle 1,4,20\rangle$ &
36. $\langle 1,4,24\rangle$ \\
\hline \rule[-2mm]{-1mm}{6mm}
37. $\langle 1,4,36\rangle$ &
38. $\langle 1,5,5\rangle$ &
39. $\langle 1,5,8\rangle$ &
40. $\langle 1,5,10\rangle$ \\
\hline \rule[-2mm]{-1mm}{6mm}
41. $\langle 1,5,25\rangle$ &
42. $\langle 1,5,40\rangle$ &
43. $\langle 1,6,6\rangle$ &
44. $\langle 1,6,9\rangle$ \\
\hline \rule[-2mm]{-1mm}{6mm}
45. $\langle 1,6,16\rangle$ &
46. $\langle 1,6,18\rangle$ &
47. $\langle 1,6,24\rangle$ &
48. $\langle 1,8,8\rangle$ \\
\hline \rule[-2mm]{-1mm}{6mm}
49. $\langle 1,8,16\rangle$ &
50. $\langle 1,8,24\rangle$ &
51. $\langle 1,8,32\rangle$ &
52. $\langle 1,8,40\rangle$ \\
\hline \rule[-2mm]{-1mm}{6mm}
53. $\langle 1,8,64\rangle$ &
54. $\langle 1,9,9\rangle$ &
55. $\langle 1,9,12\rangle$ &
56. $\langle 1,9,24\rangle$ \\
\hline \rule[-2mm]{-1mm}{6mm}
57. $\langle 1,10,30\rangle$ &
58. $\langle 1,12,12\rangle$ &
${}^{\dag}$59. $\langle 1,12,24\rangle$ &
60. $\langle 1,12,36\rangle$ \\
\hline \rule[-2mm]{-1mm}{6mm}
61. $\langle 1,16,16\rangle$ & 
62. $\langle 1,16,24\rangle$ & 
${}^{\dag}$63. $\langle 1,16,32\rangle$ & 
64. $\langle 1,16,48\rangle$ \\
\hline \rule[-2mm]{-1mm}{6mm}
${}^{\dag}$65. $\langle 1,16,144\rangle$ &
66. $\langle 1,24,24\rangle$ &
67. $\langle 1,24,72\rangle$ & 
68. $\langle 1,40,120\rangle$ \\
\hline \rule[-2mm]{-1mm}{6mm}
69. $\langle 1,48,144\rangle$ & 
70. $\langle 2,2,3\rangle$ &
71. $\langle 2,3,3\rangle$ &
72. $\langle 2,3,6\rangle$ \\
\hline \rule[-2mm]{-1mm}{6mm}
73. $\langle 2,3,8\rangle$ & 
74. $\langle 2,3,9\rangle$ & 
75. $\langle 2,3,12\rangle$ &
76. $\langle 2,3,18\rangle$ \\
\hline \rule[-2mm]{-1mm}{6mm}
77. $\langle 2,3,48\rangle$ & 
78. $\langle 2,5,6\rangle$ & 
79. $\langle 2,5,10\rangle$ & 
80. $\langle 2,5,15\rangle$ \\
\hline \rule[-2mm]{-1mm}{6mm}
81. $\langle 2,6,9\rangle$ &
82. $\langle 2,6,15\rangle$ & 
83. $\langle 3,3,4\rangle$ & 
${}^{\dag \dag}$84. $\langle 3,3,7\rangle$ \\
\hline \rule[-2mm]{-1mm}{6mm}
85. $\langle 3,3,8\rangle$ &
86. $\langle 3,4,4\rangle$ &
${}^{\dag}$87. $\langle 3,4,8\rangle$ & 
88. $\langle 3,4,12\rangle$ \\
\hline \rule[-2mm]{-1mm}{6mm} 
89. $\langle 3,4,36\rangle$ & 
90. $\langle 3,8,8\rangle$ &
91. $\langle 3,8,12\rangle$ &
92. $\langle 3,8,24\rangle$ \\ 
\hline \rule[-2mm]{-1mm}{6mm}
93. $\langle 3,8,48\rangle$ & 
94. $\langle 3,8,72\rangle$ & 
95. $\langle 3,10,30\rangle$ &
96. $\langle 3,16,48\rangle$ \\
\hline \rule[-2mm]{-1mm}{6mm}
97. $\langle 3,40,120\rangle$ &
${}^{\dag}$98. $\langle 5,6,9\rangle$ &
99. $\langle 5,6,15\rangle$ &
100. $\langle 5,8,24\rangle$ \\
\hline \rule[-2mm]{-1mm}{6mm}
101. $\langle 5,8,40\rangle$ &
102. $\langle 8,9,24\rangle$ &
103. $\langle 8,15,24\rangle$ & \\
\hline
\end{tabular}
\label{tableu3}
\end{table}

For $j=1,2,\dots,103$, let $J(j)$ be the $j$-th diagonal ternary $\z$-lattice in Table \ref{tableu3}.
We will see that every candidate for the ternary section of a regular diagonal $\z$-lattice of rank greater than or equal to 4 does appear in Table \ref{tableu3}.
Before that, we analyse the list of diagonal ternary $\z$-lattices $\{ J(j) : 1\le j\le 103\}$.
Among 103 lattices, there are exactly six irregular diagonal ternary $\z$-lattices
$$
J(j),\ \ j\in \{35,59,63,65,87,98\},
$$
marked with $\dag$ in Table \ref{tableu3}.
For each of the six $\z$-lattices $J$, the smallest positive integer which is represented locally but not globally by $J$ will be given in Table \ref{tableirr}.

%%%%%%%%%%%%%%%%%%%%%%%%%%%%%%%%%%%%%%%%%%%%%%%%%%%%%%%%%%%%%%%%%%%%%%%%%%%%%

\begin{table}[ht]
\caption{$t(J)=\psi(Q(\gen(J)),J)$ for some irregular $\z$-lattice $J$}
\vskip -10pt
\begin{tabular}{|c|c|c|c|c|c|c|}
\hline
$J$ & $\langle 1,4,20\rangle$ & $\langle 1,12,24\rangle$ & $\langle 1,16,32\rangle$ & $\langle 1,16,144\rangle$ & $\langle 3,4,8\rangle$ & $\langle 5,6,9\rangle$ \\ 
\hline
$t(J)$ & 77 & 69 & 161 & 473 & 23 & 17 \\
\hline
\end{tabular}

\label{tableirr}
\end{table}

%%%%%%%%%%%%%%%%%%%%%%%%%%%%%%%%%%%%%%%%%%%%%%%%%%%%%%%%%%%%%%%%%%%%%%%%%%%%%

On the other hand, there are exactly six regular diagonal ternary $\z$-lattices $J$ satisfying $Q(J_7)\neq \z_7$ (for this, see \cite{J} or \cite{JP}).
For each of the six $\z$-lattices $J$, we define $u(J)$ to be the smallest positive integer which is not represented by $J_7$ and is represented by $J_q$ for every prime $q\neq 7$. We provide the values $u(J)$ in Table \ref{tablereg7}.
By Lemma \ref{lembound}, a necessary condition for $J=\langle a_1,a_2,a_3\rangle$ being a ternary section of a regular diagonal $\z$-lattice $L=\langle a_1,a_2,\dots,a_k\rangle$ ($a_1\le a_2\le \cdots \le a_k$) of rank $k\ge 4$ is that $a_3\le u(J)$.
Hence $\langle 3,3,7\rangle$ passed the qualification and appear in Table \ref{tableu3} as $J(84)$, marked with $\dag \dag$, and the other five $\z$-lattices are disqualified from being candidates for the ternary section of a regular diagonal $\z$-lattice of rank $\ge 4$.

%%%%%%%%%%%%%%%%%%%%%%%%%%%%%%%%%%%%%%%%%%%%%%%%%%%%%%%%%%%%%%%%%%%%%%%%%%%%%

\begin{table}[ht]
\caption{$u(J)$ for some regular diagonal ternary $\z$-lattice $J$}
\vskip -10pt
\begin{tabular}{|c|c|c|c|c|c|c|}
\hline
$J$ & $\langle 1,1,21\rangle$ & $\langle 1,9,21\rangle$ & $\langle 1,21,21\rangle$ & $\langle 3,3,7\rangle$ & $\langle 3,7,7\rangle$ & $\langle 3,7,63\rangle$ \\ 
\hline
$u(J)$ & 7 & 7 & 3 & 21 & 1 & 1 \\
\hline
\end{tabular}

\label{tablereg7}
\end{table}

%%%%%%%%%%%%%%%%%%%%%%%%%%%%%%%%%%%%%%%%%%%%%%%%%%%%%%%%%%%%%%%%%%%%%%%%%%%%%

\begin{lem} \label{lemternary}
Let $L=\langle a_1,a_2,\dots,a_k\rangle$ be a regular diagonal $\z$-lattice of rank $k\ge 4$ such that $a_1\le a_2\le \cdots \le a_k$.
Then the ternary section $\langle a_1,a_2,a_3\rangle$ of $L$ appears in Table \ref{tableu3}.
\end{lem}

\begin{proof}
Since $L$ is primitive, for each $q\in \{2,3,5\}$, there is an index $i_q$ with $1\le i_q\le k$ such that $(q,a_{i_q})=1$.
There is a triple $(\epsilon_2,\epsilon_3,\epsilon_5)\in A$ such that $\epsilon_q\in a_{i_q}(\z_q^{\times})^2$ for $q=2,3,5$.
Then every integer in the set $S(\epsilon_2,\epsilon_3,\epsilon_5)$ is represented by $L_q$ for every $q\in \{2,3,5\}$.
Hence by Lemma \ref{lembound}, we have
$$
a_1\le s_1(\epsilon_2,\epsilon_3,\epsilon_5),\ \ \text{and}\ \ a_2\le \psi(S(\epsilon_2,\epsilon_3,\epsilon_5), \langle a_1\rangle).
$$
This implies that $(a_1,a_2)\in U_2(\epsilon_2,\epsilon_3,\epsilon_5)$.

For $q=2,3,5$, we note that every integer that is represented by $\langle a_1,a_2\rangle$ over $\z_q$ is also represented by $L$ over $\z_q$.
Furthermore, if $\left(\dfrac{\epsilon_q}{q}\right) \neq \left(\dfrac{a_i}{q}\right)$ for all $i=1,2$,
then it follows that $i_q\neq 1,2$, and thus every integer represented by $\langle a_1,a_2,a_{i_q}\rangle$ over $\z_q$ is represented by $L$ over $\z_q$.
Therefore, we have $(a_1,a_2,a_3)\in U_3$ by Lemma \ref{lembound}.
This completes the proof.
\end{proof}

%%%%%%%%%%%%%%%%%%%%%%%%%%%%%%%%%%%%%%%%%%%%%%%%%%%%%%%%%%%%%%%%%%%%%%%%%%%%%
%%%%%%%%%%%%%%%%%%%%%%%%%%%%%%%%%%%%%%%%%%%%%%%%%%%%%%%%%%%%%%%%%%%%%%%%%%%%%
%%%%%%%%%%%%%%%%%%%%%%%%%%%%%%%%%%%%%%%%%%%%%%%%%%%%%%%%%%%%%%%%%%%%%%%%%%%%%
%%%%%%%%%%%%%%%%%%%%%%%%%%%%%%%%%%%%%%%%%%%%%%%%%%%%%%%%%%%%%%%%%%%%%%%%%%%%%
%%%%%%%%%%%%%%%%%%%%%%%%%%%%%%%%%%%%%%%%%%%%%%%%%%%%%%%%%%%%%%%%%%%%%%%%%%%%%
\section{Minimal regular diagonal $\z$-lattices of rank exceeding 3}
%%%%%%%%%%%%%%%%%%%%%%%%%%%%%%%%%%%%%%%%%%%%%%%%%%%%%%%%%%%%%%%%%%%%%%%%%%%%%
%%%%%%%%%%%%%%%%%%%%%%%%%%%%%%%%%%%%%%%%%%%%%%%%%%%%%%%%%%%%%%%%%%%%%%%%%%%%%
%%%%%%%%%%%%%%%%%%%%%%%%%%%%%%%%%%%%%%%%%%%%%%%%%%%%%%%%%%%%%%%%%%%%%%%%%%%%%
%%%%%%%%%%%%%%%%%%%%%%%%%%%%%%%%%%%%%%%%%%%%%%%%%%%%%%%%%%%%%%%%%%%%%%%%%%%%%
%%%%%%%%%%%%%%%%%%%%%%%%%%%%%%%%%%%%%%%%%%%%%%%%%%%%%%%%%%%%%%%%%%%%%%%%%%%%%

As mentioned in the introduction, regular diagonal quaternary $\z$-lattices were classified by B.M. Kim \cite{BMK}, but the paper was never published.
Furthermore, the list in \cite{BMK} needs to be corrected a bit.
For this reason, we provide a list of all regular diagonal quaternary $\z$-lattices for the convenience of the readers, which is as follows.

%%%%%%%%%%%%%%%%%%%%%%%%%%%%%%%%%%%%%%%%%%%%%%%%%%%%%%%%%%%%%%%%%%%%%%%%%%%%%

\begin{prop} \label{proprk4}
A diagonal quaternary $\z$-lattice $\langle a_1,a_2,a_3,a_4\rangle$ with $a_1\le a_2\le a_3\le a_4$ is regular if and only if it appears in Table \ref{tablerk4}.
\end{prop}

%%%%%%%%%%%%%%%%%%%%%%%%%%%%%%%%%%%%%%%%%%%%%%%%%%%%%%%%%%%%%%%%%%%%%%%%%%%%%

In Table \ref{tablerk4}, for example, a diagonal quaternary $\z$-lattice $\langle 1,1,1,a_4\rangle$ is regular if and only if
$$
a_4\in \{ 2^{r-1} : r\in \n \} \cup \{ 2^{2r-1}3 : r\in \n \} \cup \{3,5,7\}.
$$

\begin{longtable}[h]{|ccc|ll|}
\hline
$a_1$ & $a_2$ & $a_3$ & $a_4$ & \\
\hline
1&1&1& $2^{r-1}, 2^{2r-1} 3$, & 3,5,7\\
\hline
1&1&2& $2^r, 2^{2r-2} 3$, & 5,6,7,9,10,11,13,14\\
\hline
1&1&3& $3^r, 2^1 3^r, 2^2 3^{2r-2}, 3^{2r-2} 5$ & \\
\hline
1&1&4& $2^{r+1}, 2^{2r+1} 3$, & 12,20,28\\
\hline
1&1&5& $2^{2r+1}$ & \\
\hline
1&1&6& $3^{r+1}, 2^1 3^{2r}$ & \\
\hline
1&1&8& $2^{r+2}, 2^{2r+2} 3$, & 24,40,56\\  
\hline
1&1&12& $2^2 3^{2r}$ & \\
\hline
1&1&16& $2^{r+3}, 2^{2r+3} 3$, & 48,80,112\\
\hline
1&1&24& $2^3 3^{2r}$ & \\
\hline
1&2&2& $2^r, 2^{2r-1} 3$, & 3,5,7\\
\hline
1&2&3& $2^{r+1}$, & 3,5,6,7,9,10\\
\hline
1&2&4& $2^{r+1}, 2^{2r} 3$, & 5,6,7,9,10,11,13,14\\
\hline
\multirow{2}{*}{1}&\multirow{2}{*}{2}&\multirow{2}{*}{5}& $5^{2r}, 2^1 5^r, 2^2 5^{2r}, 2^3 5^{2r-2}, 3^1 5^{2r}$, & \\ &&& $3^2 5^{2r-2}, 5^{2r-2} 7, 2^1 3^1 5^{2r-2}$ & \\
\hline
1&2&6& $2^{r+2}$ & \\
\hline
1&2&8& $2^{r+2}, 2^{2r+1} 3$ & \\
\hline
1&2&16& $2^{r+3}, 2^{2r+2} 3$ & \\
\hline
1&2&32& $2^{r+4}, 2^{2r+3} 3$ & \\
\hline
1&3&3& $3^r, 2^1 3^r, 2^2 3^{2r-1}, 3^{2r-1} 5$ & \\
\hline
1&3&4& $2^2 3^{r-1}, 2^3 3^{2r-2}$ & \\
\hline
1&3&6& $2^r 3$, & 9,15,18,21,27,30\\
\hline
1&3&9& $3^{r+1}, 2^1 3^{r+1}, 2^2 3^{2r}, 3^{2r} 5$ & \\
\hline
1&3&12& $2^2 3^r, 2^3 3^{2r-1}$ & \\
\hline
1&3&36& $2^2 3^{r+1}, 2^3 3^{2r}$ & \\
\hline
1&4&4& $2^{r+1}, 2^{2r+1} 3$, & 12,20,28\\
\hline
1&4&8& $2^{r+2}, 2^{2r} 3$, & 20,24,28,36,40,44,52,56\\
\hline
1&4&12& $2^2 3^r, 2^3 3^r, 2^4 3^{2r-2}, 2^2 3^{2r-2} 5$ & \\
\hline
1&4&16& $2^{r+3}, 2^{2r+3} 3$, &48,80,112\\
\hline
$\dag$1&4&20& & 32 \\
\hline
1&4&24& $2^2 3^{r+1}, 2^3 3^{2r}$ & \\
\hline
1&5&5& $2^{2r-1} 5$, & 5,15\\
\hline
1&5&8& $2^3 5^{r-1}, 2^4 5^{2r-2}, 2^5 5^{2r-2}, 2^3 3^1 5^{2r-2}$ & \\
\hline
\multirow{2}{*}{1}&\multirow{2}{*}{5}&\multirow{2}{*}{10}& $5^{2r+1}, 2^1 5^r, 2^2 5^{2r-1}, 2^3 5^{2r-1}$, & \\ &&& $3^1 5^{2r-1}, 3^2 5^{2r-1}, 5^{2r-1} 7, 2^1 3^1 5^{2r-1}$ & \\
\hline
1&5&40& $2^3 5^r, 2^4 5^{2r-1}, 2^5 5^{2r-1}, 2^3 3^1 5^{2r-1}$ & \\
\hline
1&6&9& $3^{r+1}, 2^1 3^{2r}$ & \\
\hline
1&8&8& $2^{r+2}, 2^{2r+1} 3$ & \\
\hline
1&8&16& $2^{r+3}, 2^{2r+2} 3$, & 24,40,56\\
\hline
1&8&24& $2^{r+4}$ & \\
\hline
1&8&32& $2^{r+4}, 2^{2r+3} 3$ & \\
\hline
1&8&64& $2^{r+5}, 2^{2r+4} 3$ & \\
\hline
1&9&12& $2^2 3^{2r}$ & \\
\hline
1&9&24& $2^3 3^{2r}$ & \\
\hline
1&12&12& $2^2 3^r, 2^3 3^{2r-1}$ & \\
\hline
$\dag$1&12&24& & 24,36,48,60\\
\hline
1&12&36& $2^2 3^{r+1}, 2^3 3^{r+1}, 2^4 3^{2s}, 2^2 3^{2s} 5$ & \\
\hline
1&16&16& $2^{r+3}, 2^{2r+3} 3$, & 48,80,112\\
\hline
\multirow{2}{*}{$\dag$1}&\multirow{2}{*}{16}&\multirow{2}{*}{32}& & 32,48,64,80,96,\\ &&&& 112,128,144,160\\
\hline
1&16&48& $2^4 3^{2r}, 2^5 3^{2r}$ & \\
\hline
1&48&144& $2^4 3^{2r}, 2^5 3^{2r}$ & \\
\hline
2&3&3& $3^r, 2^1 3^{2r-1}$ & \\
\hline
2&3&6& $2^r 3$, & 9,15\\
\hline
2&3&9& & 9,18\\
\hline
3&3&4& $2^2 3^{2r-1}$ & \\
\hline
3&3&8& $2^3 3^{2r-1}$ & \\
\hline
3&4&4& $2^2 3^{r-1}, 2^3 3^{2r-2}$ & \\
\hline
$\dag$3&4&8& & 8,12,16,20\\
\hline
3&4&12& $2^2 3^r, 2^3 3^r, 2^4 3^{2r-1}, 2^2 3^{2r-1} 5$ & \\
\hline
3&4&36& $2^2 3^{r+1}, 2^3 3^{2r}$ & \\
\hline
3&8&12& $2^2 3^r, 2^3 3^{2r-1}$ & \\
\hline
3&8&24& $2^{r+2} 3$ & \\
\hline
3&16&48& $2^4 3^{2r-1}, 2^5 3^{2r-1}$ & \\
\hline
\caption{{\small Regular diagonal quaternary $\z$-lattices $\langle a_1,a_2,a_3,a_4\rangle$ ($r\in \n$)}}
\label{tablerk4}
\end{longtable}

%%%%%%%%%%%%%%%%%%%%%%%%%%%%%%%%%%%%%%%%%%%%%%%%%%%%%%%%%%%%%%%%%%%%%%%%%%%%%

As shown in \cite{BMK}, the regularity of any of the $\z$-lattices in Table \ref{tablerk4} can be shown by a routine method.
Even so, the $\z$-lattices
$$
\langle 2,3,6,a_4\rangle, \quad a_4\in \{9,15\} \cup \{2^{2r}3 : r\in \n \}
$$
do not appear in \cite{BMK} as mentioned before, we show their regularities as a sample.

\begin{prop}
For any positive integer $a_4\in \{9,15\} \cup \{2^{2r}3 : r\in \n \}$, the diagonal quaternary $\z$-lattice $\langle 2,3,6,a_4\rangle$ is regular.
\end{prop}

\begin{proof}
Let $a_4$ be any integer in the set $\{9,15\} \cup \{2^{2r}3 : r\in \n \}$ and let $t$ be the nonnegative integer satisfying
$$
2t=\max(\ord_2(a_4)-2,0).
$$
Note that the class number of $\langle 2,3,6\rangle$ is one and
\begin{equation} \label{eq236}
\n-Q(\langle 2,3,6\rangle)=\{ 2^{2s}(8u+7) : s,u\in \n \cup \{0\} \} \cup \{ 3v+1 : v\in \n \cup \{0\} \}.
\end{equation}
One may easily check that
$$
Q(\gen(\langle 2,3,6,a_4\rangle)-Q(\langle 2,3,6\rangle)=\{2^{2s}\cdot 2^{2t}(24u+v) : s,u\in \n \cup \{0\}, v\in \{15,23\} \}.
$$
Hence, it suffices to show that
$$
2^{2t}(24u+v)\ra \langle 2,3,6,a_4\rangle
$$
for any nonnegative integer $u$ and $v\in \{15,23\}$.
If $a_4\in \{9\} \cup \{2^{2r}3 : r\in \n \}$, then one may easily see that
$$
2^{2t}(24u+v)-a_4\ra \langle 2,3,6\rangle
$$
by using Equation \eqref{eq236}.
From this, it follows that $2^{2t}(24u+v)$ is represented by $\langle 2,3,6,a_4\rangle$.
If $a_4=15$, then we have $t=0$ by definition and one may see that
$$
24u+v-15\cdot 2^2\ra \langle 2,3,6\rangle,
$$
for $u\ge 2$.
One may directly check that both $v$ and $24+v$ are represented by $\langle 2,3,6,15\rangle$.
This completes the proof.
\end{proof}

%%%%%%%%%%%%%%%%%%%%%%%%%%%%%%%%%%%%%%%%%%%%%%%%%%%%%%%%%%%%%%%%%%%%%%%%%%%%%

We are ready to classify all regular diagonal $\z$-lattices of rank exceeding 3.
As shown in Lemma \ref{lemternary}, there are 103 possibilities for the ternary section of such a $\z$-lattice.
We will find all minimal regular diagonal $\z$-lattices $L=\langle a_1,a_2,a_3,\dots,a_k\rangle$ having $J(j)=\langle a_1,a_2,a_3\rangle$ as its ternary section for each $j=1,2,\dots,103$.
We divide 103 $j$'s into six batches as in Table \ref{tablebatch}.

%%%%%%%%%%%%%%%%%%%%%%%%%%%%%%%%%%%%%%%%%%%%%%%%%%%%%%%%%%%%%%%%%%%%%%%%%%%%%

\begin{table}[ht]
\caption{Batches for 103 candidates for the ternary section of $L$}
\begin{tabular}{|c|c|c|}
\hline
Batch & $j$ & $\#$ of $j$'s\\
\hline
\multirow{2}{*}
{Batch 1} & 1,2,3,4,5,6,7,10,12,13,14,16,17,19,20, & \multirow{2}{*}{28}\\
& 21,24,30,32,34,44,48,49,50,51,53,61,71 & \\
\hline
\multirow{2}{*}{Batch 2}&9,11,22,23,26,29,33,36,38,39,42,55,56, &\multirow{2}{*}{27}\\
&58,60,64,69,72,74,83,85,86,88,89,91,92,96 & \\
\hline
Batch 3 & 35,59,63,87 & 4\\
\hline
Batch 4 & 15,40 & 2\\
\hline
\multirow{2}{*}{Batch $2'$} & 8,18,25,27,28,31,37,41,43,45,46,47,52,54,57,62,66,67,68,70, & \multirow{2}{*}{39}\\
& 73,75,76,77,78,79,80,81,82,90,93,94,95,97,99,100,101,102,103 & \\
\hline
Batch $3'$ & 65,84,98 & 3\\
\hline
\end{tabular}
\label{tablebatch}
\end{table}

%%%%%%%%%%%%%%%%%%%%%%%%%%%%%%%%%%%%%%%%%%%%%%%%%%%%%%%%%%%%%%%%%%%%%%%%%%%%%

Since all cases in a batch can be dealt with in a similar manner, we look into one representative case for each batches. We let $L=\langle a_1,a_2,a_3,a_4,\dots,a_k\rangle$ $(a_1\le a_2\le \cdots \le a_k$)
be a minimal regular diagonal $\z$-lattice of rank $k\ge 4$ having $J(j)=\langle a_1,a_2,a_3\rangle$ as its ternary section.

\noindent ({\bf Batch 1}) $j=1$ and $L=\langle 1,1,1,a_4,a_5,\dots,a_k\rangle$.
Assume first that $\ord_2(a_4)\le 2$.
Then one may easily check that
$$
\n \cup \{0\}=Q(\gen(\langle 1,1,1,a_4\rangle))\subseteq Q(\gen(L))=Q(L).
$$
Since the positive integer 7 is not represented by $\langle 1,1,1\rangle$, we have $1\le a_4\le 7$.
Note that $\langle 1,1,1,a_4\rangle$ is universal for every $a_4=1,2,\dots,7$.
Thus we have $k\le 4$ since otherwise we would have $a_5$ redundant and this contradicts to the minimality of $L$.
Next, assume that $\ord_2(a_4)\in \{2s-1,2s\}$ for some integer $s\ge 2$.
One may easily show that $7\cdot 2^{2s-2}$ is locally represented by $\langle 1,1,1,a_4\rangle$.
Thus we have
$$
7\cdot 2^{2s-2}\in Q(\gen(\langle 1,1,1,a_4\rangle))\subseteq Q(\gen(L))=Q(L).
$$
Since $7\cdot 2^{2s-2}\nra \langle 1,1,1\rangle$, we have $a_4\le 7\cdot 2^{2s-2}$.
This implies that
$$
a_4\in \{2^{2s-1},2^{2s},3\cdot 2^{2s-1}\}.
$$
By Proposition \ref{proprk4}, the diagonal quaternary $\z$-lattice $\langle 1,1,1,a_4\rangle$ is regular.
Now suppose that $k\ge 5$.
If $\ord_2(a_5)\ge 2s-1$, then one may easily check that $a_5$ is redundant by using Lemmas \ref{lemredodd} and \ref{lemred2}.
Thus we may further assume that $\ord_2(a_5)\le 2s-2$.
Using a similar argument with the case of $a_4$, one may easily show that
$$
\begin{cases}
1\le a_5\le 7&\text{if}\ \ \ord_2(a_5)\le 2,\\
a_5\in \{2^{2t-1},2^{2t},3\cdot 2^{2t-1}\}&\text{if}\ \ \ord_2(a_5)\in \{2t-1,2t\},\ \ 2\le t\le s-1.\end{cases}
$$
Then it follows that $\langle 1,1,1,a_5\rangle$ is regular and $a_4$ is redundant.
This is absurd, and we have $k\le 4$.

In summary, if $L=\langle 1,1,1,a_4,a_5,\dots,a_k\rangle$ is a minimal regular diagonal $\z$-lattice of rank $k\ge 4$, then we have
$$
k=4\ \ \text{and}\ \ a_4\in \{ 2^{r-1} : r\in \n\} \cup \{2^{2r-1}3 : r\in \n \} \cup \{3,5,7\}.
$$

\noindent ({\bf Batch 2}) $j=9$ and $L=\langle 1,1,12,a_4,a_5,\dots,a_k\rangle$.
Note that $\ord_2(a_i)\ge 2$ for every $i=4,5,\dots,k$, since otherwise we have
$$
3\in Q(\gen(\langle 1,1,12,a_i\rangle))\subseteq Q(\gen(L)),
$$
and this is absurd because 3 cannot be represented by $L$.
One may easily show that $\ord_3(a_i)\ge 2$ for every $i=4,5,\dots,k$, in a similar manner.
Now, we let $\ord_3(a_4)\in \{2s,2s+1\}$ for some positive integer $s$.
Then one may easily check that
$$
2\cdot 3^{2s+1}\in Q(\gen(\langle 1,1,12,a_i\rangle))\subseteq Q(\gen(L))=Q(L).
$$
Since $2\cdot 3^{2s+1}\nra \langle 1,1,12\rangle$, we have $a_4\le 2\cdot 3^{2s+1}$.
From these, one may easily deduce that $a_4=4\cdot 3^{2s}$.
By Proposition \ref{proprk4}, $\langle 1,1,12,a_4\rangle$ is regular.
Suppose that $k\ge 5$.
Since $\ord_3(a_5)$ must be greater than or equal to 2, we have $\ord_3(a_5)\in \{2t,2t+1\}$ for some positive integer $t$.
If $s\le t$, then $a_5$ is redundant.
If $s>t$, then one may easily deduce that $a_5=4\cdot 3^{2t}$ as above, and this implies that $\langle 1,1,12,a_5\rangle$ is regular and $a_4$ is redundant.

In summary,  if $L=\langle 1,1,12,a_4,a_5,\dots,a_k\rangle$ is a minimal regular diagonal $\z$-lattice of rank $k\ge 4$, then we have
$$
k=4\ \ \text{and}\ \ a_4\in \{ 2^23^{2r} : r\in \n \}.
$$

\noindent ({\bf Batch 3}) $j=35$ and $L=\langle 1,4,20,a_4,a_5,\dots,a_k\rangle$.
As you may see in Table \ref{tableirr}, one may directly check that
$$
\psi(Q(\gen(\langle 1,4,20\rangle)),\langle 1,4,20\rangle)=77.
$$
Thus we have $a_4\le 77$.
One may easily check that
$$
\psi(Q(\gen(\langle 1,4,20,u\rangle)),\langle 1,4,20,u\rangle)<u\ \ \text{for every}\ \ u\in \{20,21,22,\dots,77\}-\{32\}.
$$
By Lemma \ref{lembound}, the only possibility is that $a_4=32$.
Then $\langle 1,4,20,a_4\rangle$ is regular by Proposition \ref{proprk4}.
Suppose that $k\ge 5$.
If $\ord_2(a_5)\ge 5$, then one may easily show that $a_5$ is redundant.
If $\ord_2(a_5)\le 4$, then one may easily check that
$$
12\in Q(\gen(\langle 1,4,20,32,a_5\rangle))\subseteq Q(\gen(L)),
$$
and this is absurd since 12 cannot be represented by $L$.

In summary,  if $L=\langle 1,4,20,a_4,a_5,\dots,a_k\rangle$ is a minimal regular diagonal $\z$-lattice of rank $k\ge 4$, then we have
$$
k=4\ \ \text{and}\ \ a_4=32.
$$

\noindent ({\bf Batch 4}) $j=15$ and $L=\langle 1,2,5,a_4,a_5,\dots,a_k\rangle$.
Let $\ord_5(a_4)\in \{2s,2s+1\}$ for some $s\ge 0$.
Then one may easily see that
$$
10\cdot 5^{2s}=2\cdot 5^{2s+1}\in Q(\gen(\langle 1,2,5,a_4\rangle))\subseteq Q(\gen(L))=Q(L).
$$
Since $10\cdot 5^{2s}\nra \langle 1,2,5\rangle$, we have
$$
a_4\in \{5^{2s},2\cdot 5^{2s},3\cdot 5^{2s},\dots, 10\cdot 5^{2s}\}.
$$

First, assume that $a_4\in \{6,7,8,9,10\}$.
Then since $\langle 1,2,5,a_4\rangle$ is universal, we have $k\le 4$ by the minimality of $L$.

Second, assume that
$$
a_4\in \{5^{2s},2\cdot 5^{2s},3\cdot 5^{2s},4\cdot 5^{2s},6\cdot 5^{2s},7\cdot 5^{2s},8\cdot 5^{2s},9\cdot 5^{2s},10\cdot 5^{2s}\},\ \ s\in \n.
$$
Then $\langle 1,2,5,a_4\rangle$ is regular by Proposition \ref{proprk4}.
Suppose that $k\ge 5$.
If $\ord_5(a_5)\le 1$, then one may easily check that
$$
10\in Q(\gen(\langle 1,2,5,a_5\rangle))\subseteq Q(\gen(L)).
$$
However, $L$ cannot represent 10 since $10\nra \langle 1,2,5\rangle$ and $10<a_4$.
This is absurd and thus we have $\ord_5(a_5)\in \{2t,2t+1\}$ for some $t\in \n$.
If $t\ge s$, then one may easily show that $a_5$ is redundant.
If $t<s$, then one may easily deduce that $2\cdot 5^{2t+1}$ is locally represented by $L$ but not globally by $L$.
Hence $k\le 4$ in the second case.

Last, assume that $a_4=5^{2s+1}$ for some nonnegative integer $s$.
Since the ternary $\z$-lattice $\langle 1,2,5\rangle$ is regular, every nonnegative integer not of the form $5^{2u+1}(5v\pm 2)$ for some nonnegative integers $u$ and $v$ is represented by $\langle 1,2,5\rangle$.
Using this fact, one may easily check that
$$
Q(\gen(\langle 1,2,5,5^{2s+1}\rangle))-Q(\langle 1,2,5,5^{2s+1}\rangle)=\{3\cdot 5^{2s+1}\}.
$$
Thus we have $a_5\le 3\cdot 5^{2s+1}$.
Suppose that $\ord_5(a_5)\le 2s-1$.
Then one may easily check that
$$
2\cdot 5^{2s-1}\in Q(\gen(\langle 1,2,5,a_5\rangle))\subseteq Q(\gen(L)).
$$
However, $2\cdot 5^{2s-1}$ cannot be represented by $L$ since it is not represented by $\langle 1,2,5\rangle$ and $2\cdot 5^{2s-1}<a_4$.
This is absurd and thus we have $\ord_5(a_5)\ge 2s$.
Since
$$
2s\le \ord_5(a_5)\ \ \text{and}\ \ a_4=5^{2s+1}\le a_5\le 3\cdot 5^{2s+1},
$$
we have $a_5\in \{5^{2s+1},11\cdot 5^{2s},12\cdot 5^{2s},13\cdot 5^{2s},14\cdot 5^{2s},15\cdot 5^{2s} \}$.
From this, it follows that
$$
3\cdot 5^{2s+1}-a_5\ra \langle 1,2,5,5^{2s+1}\rangle.
$$
Hence we have
\begin{align*}
Q(\gen(\langle 1,2,5,5^{2s+1}\rangle))&=Q(\langle 1,2,5,5^{2s+1}\rangle) \cup \{3\cdot 5^{2s+1}\} \\
&\subset Q(\langle 1,2,5,5^{2s+1},a_5\rangle).
\end{align*}
Since $2s\le \ord_5(a_5)$, we have
\begin{align*}
Q(\gen(\langle 1,2,5,5^{2s+1},a_5\rangle))&=Q(\gen(\langle 1,2,5,5^{2s+1}\rangle))\\
&\subset Q(\langle 1,2,5,5^{2s+1},a_5\rangle).
\end{align*}
Thus $\langle 1,2,5,5^{2s+1},a_5\rangle$ is regular.
Suppose that $k\ge 6$.
If $\ord_5(a_6)\ge 2s$, then one may easily deduce that $a_6$ is redundant.
If $\ord_5(a_6)\le 2s-1$, then one may easily show that
$$
2\cdot 5^{2s-1}\in Q(\gen(\langle 1,2,5,a_6\rangle))\subseteq Q(\gen(L)),
$$
but $2\cdot 5^{2s-1}$ cannot be represented by $L$.
Hence, we have $k\le 5$.

In summary, if $L=\langle 1,2,5,a_4,a_5,\dots,a_k\rangle$ is a minimal regular diagonal $\z$-lattice of rank $k\ge 4$, then we have
$$
k=4\ \ \text{and}\ \ a_4\in \{6,7,8,9,10\} \cup \{l\cdot 5^{2r} : l\in \{1,2,3,4,6,7,8,9,10\},\ r\in \n \},
$$
or
$$
k=5\ \ \text{and}\ \ (a_4,a_5)\in \{(5^{2r+1},l\cdot 5^{2r}) : l\in \{ 5,11,12,13,14,15\},\ r\in \n \}.
$$

\noindent ({\bf Batch} ${\bf 2'}$) $j=8$ and $L=\langle 1,1,9,a_4,a_5,\dots,a_k\rangle$.
Assume first that $\ord_2(a_4)\le 2$.
Then one may easily show that
$$
7\in Q(\gen(\langle 1,1,9,a_4\rangle))\subseteq Q(\gen(L)).
$$
This is absurd since 7 cannot be represented by $L$.
So we have $\ord_2(a_4)\ge 3$ and thus $\ord_2(a_4)\in \{2s-1,2s\}$ for some $r\ge 2$.
Then one may easily check that
$$
7\cdot 2^{2s-2}\in Q(\gen(\langle 1,1,9,a_4\rangle))\subseteq Q(\gen(L))=Q(L).
$$
Since $7\cdot 2^{2s-2}\nra \langle 1,1,9\rangle$, we have $a_4\le 7\cdot 2^{2s-2}$, and thus $a_4\in \{2^{2s-1},2^{2s},3\cdot 2^{2s-1}\}$.
Then one may easily check that
$$
3\in Q(\gen(\langle 1,1,9,a_4\rangle))\subseteq Q(\gen(L)),
$$
which is absurd since 3 cannot be represented by $L$.

In summary, there is no minimal regular diagonal $\z$-lattice $\langle a_1,a_2,\dots,a_k\rangle$ of rank $k\ge 4$ with $a_1\le a_2\le \cdots \le a_k$ such that $(a_1,a_2,a_3)=(1,1,9)$.

\noindent ({\bf Batch} ${\bf 3'}$) $j=65$ and $L=\langle 1,16,144,a_4,a_5,\dots,a_k\rangle$.
We first assert that
\begin{equation} \label{eq116144}
a_i\equiv 0\Mod {2^7 3^2},\ \ i=4,5,\dots,k.
\end{equation}
Let $i\in \{4,5,\dots,k\}$.
Suppose that $\ord_2(a_i)\le 6$.
Then one may easily check that
$$
112\in Q(\gen(\langle 1,16,144,a_i))\subseteq Q(\gen(L)).
$$
This is absurd since 112 cannot be represented by $L$.
Suppose that $\ord_3(a_4)\le 1$.
Then at least one of 3 and 6 is represented by $\langle 1,16,144,a_i\rangle$ over $\z_3$.
Using this, one may easily check that 48 or 33 is locally represented by $\langle 1,16,144,a_i\rangle$.
This is absurd since any of the two integers 48 and 33 cannot be represented by $L$.
Hence we showed \eqref{eq116144}, and in particular, $a_4\ge 2^7 3^2=1152$.
However, from Lemma \ref{lembound} and
$$
\psi(Q(\gen(\langle 1,16,144\rangle)),\langle 1,16,144\rangle)=473,
$$
it follows that $a_4\le 473$, which is absurd.

In summary, there is no minimal regular diagonal $\z$-lattice $\langle a_1,a_2,\dots,a_k\rangle$ of rank $k\ge 4$ with $a_1\le a_2\le \cdots \le a_k$ such that $(a_1,a_2,a_3)=(1,16,144)$.

%%%%%%%%%%%%%%%%%%%%%%%%%%%%%%%%%%%%%%%%%%%%%%%%%%%%%%%%%%%%%%%%%%%%%%%%%%%%%
\vskip 10pt
We summarize the discussion as follow.
If $j$ is in one of Batches 1-3, then $\rk (L)=4$ and $L$ appears in Table \ref{tablerk4}.
If $j$ is in Batch 4, then $\rk (L)\in \{4,5\}$ and $L$ appears in Table \ref{tablerk4} or Table \ref{tablerk5}.
For the values $j$ in Batch $2'$ and $3'$, such an $L$ cannot exist.
Hence we have 

%%%%%%%%%%%%%%%%%%%%%%%%%%%%%%%%%%%%%%%%%%%%%%%%%%%%%%%%%%%%%%%%%%%%%%%%%%%%%

\begin{thm}
Let $L=\langle a_1,a_2,\dots,a_k\rangle$ be a diagonal $\z$-lattice of rank $k\ge 4$ such that $a_1\le a_2\le \cdots \le a_k$.
Then $L$ is minimal regular if and only if $k\in \{4,5\}$ and it appears in Table \ref{tablerk4} or Table \ref{tablerk5}.
\end{thm}

%%%%%%%%%%%%%%%%%%%%%%%%%%%%%%%%%%%%%%%%%%%%%%%%%%%%%%%%%%%%%%%%%%%%%%%%%%%%%

\begin{table}[ht]
\caption{Minimal regular diagonal quinary $\z$-lattices $\langle a_1,a_2,a_3,a_4,a_5\rangle$ ($r\in \n$)}
\begin{tabular}{|ccc|l|}
\hline
$a_1$ & $a_2$ & $a_3$ & $(a_4,a_5)$\\
\hline
1 & 2 & 5 & $(5^{2r-1},5^{2r-2}s)$, $s=5,11,12,13,14,15$\\
\hline
1 & 5 & 10 & $(5^{2r},5^{2r-1}s)$, $s=5,11,12,13,14,15$\\
\hline
\end{tabular}
\label{tablerk5}
\end{table}

%%%%%%%%%%%%%%%%%%%%%%%%%%%%%%%%%%%%%%%%%%%%%%%%%%%%%%%%%%%%%%%%%%%%%%%%%%%%%

Among the minimal regular diagonal quinary $\z$-lattices, which are listed in Table \ref{tablerk5}, some are new and some are not.
To see this, we put
$$
K(1,r,s)=\langle 1,2,5,5^{2r-1},5^{2r-2}s\rangle \ \ \text{and}\ \ K(2,r,s)=\langle 1,5,10,5^{2r},5^{2r-1}s\rangle,
$$
for $r=1,2,\dots$, and $s\in \{5,11,12,13,14,15\}$.
One may check that the quaternary $\z$-lattice $\langle 1,2,5,5^{2r-1}+5^{2r-1}\rangle$ is regular and
$$
Q(\langle 1,2,5,5^{2r-1}+5^{2r-1}\rangle)=Q(K(1,r,5)).
$$
Hence the minimal regular diagonal quinary $\z$-lattice $K(1,r,5)$ is not new for any positive integer $r$.
One may easily show that $K(2,r,5)$ is not new in a similar manner.

%%%%%%%%%%%%%%%%%%%%%%%%%%%%%%%%%%%%%%%%%%%%%%%%%%%%%%%%%%%%%%%%%%%%%%%%%%%%%

\begin{prop}
Under the notation given above, both regular quinary $\z$-lattices $K(1,r,s)$ and $K(2,r,s)$ are new for any positive integer $r$ and $s\in \{11,12,13,14,15\}$.
\end{prop}

\begin{proof}
Since the proofs are quite similar to each other, we only provide the proof for $K(1,1,11)$.
Put $L=K(1,1,11)$ and take a basis $\{ \mathbf{x}_i\}_{1\le i\le 5}$ of $L$ with respect to which
$$
L=\z \mathbf{x}_1+\z \mathbf{x}_2+\cdots+\z \mathbf{x}_5\simeq \langle 1,2,5,5,11\rangle.
$$
Let $\ell$ be a sublattice of $L$ such that $Q(\ell)=Q(L)$.
Since $1\in Q(\ell)$ and
$$
\{ \mathbf{x}\in L : Q(\mathbf{x})=1\}=\{ \pm \mathbf{x}_1\},
$$
we have $\mathbf{x}_1\in \ell$.
One may use
$$
\{ \mathbf{x}\in L : Q(\mathbf{x})=2\}=\{ \pm \mathbf{x}_2\}
$$
to see that $\mathbf{x}_2\in \ell$.
Since
$$
\{ \mathbf{x}\in L : Q(\mathbf{x})=5\}=\{ \pm \mathbf{x}_3,\pm \mathbf{x}_4\},
$$
we may assume that $\mathbf{x}_3\in \ell$ by interchanging $\mathbf{x}_3$ and $\mathbf{x}_4$, if necessary.
From 
$$
\{ \mathbf{x}\in L : Q(\mathbf{x})=2\cdot 5^{r-1}=10\}=\{ \pm 5^{r-1}\mathbf{x}_3\pm \mathbf{x}_4\}=\{ \pm \mathbf{x}_3\pm \mathbf{x}_4\},
$$
it follows that $\mathbf{x}_4\in \ell$.
If we let $\mathbf{x}$ be a vector in $\ell$ with $Q(\mathbf{x})=15$, then $\mathbf{x}=c_1\mathbf{x}_1+c_2\mathbf{x}_2+c_3\mathbf{x}_3+c_4\mathbf{x}_4\pm \mathbf{x}_5$ for some integers $c_1,c_2,c_3$ and $c_4$.
This clearly implies that $\mathbf{x}_5\in \ell$.
Hence we have $\mathbf{x}_i\in \ell$ for every $i=1,2,3,4,5$, which implies that $\ell=L$.
This proves that $L=K(1,1,11)$ is new.
\end{proof}

%%%%%%%%%%%%%%%%%%%%%%%%%%%%%%%%%%%%%%%%%%%%%%%%%%%%%%%%%%%%%%%%%%%%%%%%%%%%%

We conclude with the following example which describes the way how one may recover all regular diagonal $\z$-lattices of rank greater than 3 from minimal ones.

%%%%%%%%%%%%%%%%%%%%%%%%%%%%%%%%%%%%%%%%%%%%%%%%%%%%%%%%%%%%%%%%%%%%%%%%%%%%%

\begin{exam}
As in Example \ref{ex148144}, one may easily deduce that any regular diagonal $\z$-lattice which can be obtained from $K(r)=\langle 1,48,144,2^5 3^{2r}\rangle$ by redundant insertions is of the form
$$
\langle 1,48,144,2^5 3^{2r},2^4 3^{2r}b_5,2^4 3^{2r}b_6,\dots,2^4 3^{2r}b_k\rangle,\quad (b_5,b_6,\dots,b_k)\in \mathcal{N}(k-4).
$$
Note that if $b_5=1$, then this $\z$-lattice can also be obtained from
$$
L(r)=\langle 1,48,144,2^4 3^{2r}\rangle.
$$
Therefore any regular diagonal $\z$-lattice having $\langle 1,48,144\rangle$ as its ternary section is of the form
$$
\langle 1,48,144,2^4 3^{2r}c_4,2^4 3^{2r}c_5,2^4 3^{2r}c_6,\dots,2^4 3^{2r}c_k\rangle,
$$
where $(c_4,c_5,\dots,c_k)\in \mathcal{N}(k-3)$ with $c_4\in \{1,2\}$.
\end{exam}

%%%%%%%%%%%%%%%%%%%%%%%%%%%%%%%%%%%%%%%%%%%%%%%%%%%%%%%%%%%%%%%%%%%%%%%%%%%%%
Now one may explicitly write down all regular diagonal $\z$-lattices of rank $k\ge 4$.
Conversely, for a given diagonal $\z$-lattice $L$ of rank $k\ge 4$, one may effectively determine whether $L$ is regular or not.

%%%%%%%%%%%%%%%%%%%%%%%%%%%%%%%%%%%%%%%%%%%%%%%%%%%%%%%%%%%%%%%%%%%%%%%%%%%%%

\end{document}